\documentclass{article}

\usepackage{amsmath, amsthm, amssymb, amstext, amsfonts}
\usepackage{enumerate}
\usepackage{graphicx}

\theoremstyle{plain}

\newtheorem{thm}{Theorem}[section]

\newtheorem{lem}[thm]{Lemma}

\newtheorem{prop}[thm]{Proposition}
\newtheorem{Prop}[thm]{Proposition}

\newtheorem{Quest}{Question}

\theoremstyle{definition}

\newcommand{\N}{\ensuremath{\mathbb{N}}}

\newcommand{\sm}{\ensuremath{\setminus}}

\newcommand{\inv}{\ensuremath{^{-1}}}

\newcommand{\rand}{\partial}
\newcommand{\Aut}{\textnormal{Aut}}

\newcommand{\sub}{\subseteq}

\newcommand{\seq}[3]{(#1_#2)_{#2\in #3}}
\newcommand{\sequ}[1]{(#1_i)_{i\in{\mathbb N}}}

\newcommand{\comment}[1]{}

\newcommand{\nat}{{\mathbb N}}

\newcommand{\ganz}{{\mathbb Z}}

\begin{document}

\title{On fixing boundary points\\of transitive hyperbolic graphs}

\author{Agelos Georgakopoulos\thanks{Supported by FWF grant P-19115-N18.} \medskip \\
 {Technische Universit\"at Graz}\\
 {Steyrergasse 30, 8010}\\
 {Graz, Austria}\\
 \\
 \\
 Matthias Hamann \medskip \\
{Fachbereich Mathematik}\\
{Universit\"at Hamburg}\\
{Bundes\-stra\ss e~55}\\
{20146 Hamburg, Germany}\\
\\
}
\date{\today}
\maketitle

\begin{abstract}
\noindent We show that there is no $1$-ended, planar, hyperbolic graph such that the stabilizer of one of its hyperbolic boundary points acts transitively on the vertices of the graph.
This gives a partial answer to a question by Kaimanovich and Woess.
\end{abstract}

\section{Introduction}

In~\cite{Woess-Topo}, Woess asked for a classification of the multi-ended locally finite graphs such that a subgroup of their automorphism group acts transitively on the vertices and fixes an end.
This problem was solved by M\"oller~\cite{EoGII} by showing that these graphs are quasi-isometric to semi-regular trees.
For $1$-ended graphs the above question makes no sense, however it becomes interesting if one refines the ends by considering some other boundary. Kaimanovich and Woess~\cite{KaimWoess} considered this question with respect to the Gromov-hyperbolic boundary.

As the hyperbolic boundary is a refinement of the ends of a graph, an end that contains a hyperbolic boundary point fixed by a subgroup of the automorphism group of the graph is also fixed by that group.
Thus, the situation is solved in the cases where the hyperbolic graph has more than one end by M\"oller's aforementioned result, and the only case that remains to be discussed is the $1$-ended. Thus Kaimanovich and Woess~\cite{KaimWoess} asked:

\begin{Quest}{\em \cite[Section~6.4]{KaimWoess}}\label{quest_KaimWoess}
Does there exist a $1$-ended locally finite hyperbolic graph $G$ and a group acting transitively on $VG$ and fixing precisely one hyperbolic boundary point?
\end{Quest}

\medskip

Perhaps the only known result regarding Question~\ref{quest_KaimWoess} is that, as proved in~\cite[Section~4.D]{W-FixedSets}, for a finitely generated hyperbolic group, the group itself
acts transitively on the vertices of any of its locally finite Cayley graphs but fixes no hyperbolic boundary point.

The main result of this paper is that no graph satisfying the assertion of  Question~\ref{quest_KaimWoess} has an embedding into the Euclidean plane. As an intermediate step, we obtain a general result (Lemma~\ref{lem_ellipticExists}) proving the existence of certain types of automorphisms in a group as in Question~\ref{quest_KaimWoess} that might help prove the general case.

\section{Definitions and basic facts}

\subsection{Hyperbolic graphs}

In this section we define hyperbolic graphs and various  related objects.
For a more detailed introduction to hyperbolicity, we refer to~\cite{ABCFLMSS,CoornDelPapa,GhHaSur,gromov} and \cite[Chapter~22]{woessBook}. We will use the terminology of \cite{DiestelBook10}.

\medskip

Let $G=(VG,EG)$ be a graph.
A {\em geodesic} is a path between two vertices $x$ and $y$ with length $d(x,y)$, i.e.\ the $x-y$ distance in the graph.
The graph $G$ is called {\em $\delta$-hyperbolic} for a $\delta\geq 0$ if it is locally finite and if for every three vertices $x,y,z\in VG$, for every choice of three geodesics $\pi_{xy},\pi_{yz},\pi_{zx}$ joining $x,y,z$ in pairs, and for any point $\xi$ on $\pi_{xy}$, there is a point on $\pi_{yz}$ or $\pi_{zx}$ having distance at most $\delta$ to~$\xi$. Note that $\xi$ might be a vertex or an inner point of an edge\footnote{We are considering the graphs as $1$-simplices, which means that every edge is assumed to be an isometric image of the unit interval $[0,1]$.
}.
We call $G$ {\em hyperbolic} if there exists a $\delta\geq 0$ such that $G$ is $\delta$-hyperbolic.

A {\em ray} is a one-way infinite path and a {\em double ray} is a two-way infinite path.
Two rays are {\em equivalent} if for any finite set $S$ of vertices they both lie  eventually in the same component of $G-S$.
The equivalence classes of this relation are the {\em ends} of~$G$.

A ray or double ray is {\em geodetic} if every finite subpath of its is a geodesic.
Two geodetic rays  $\pi=x_1x_2\ldots$ and $\pi'=y_1y_2\ldots$ are {\em equivalent} if for every $i\in\N$ there is an $M\in \N$ such that $\liminf_{n\to\infty}d(x_{i+n},\pi')\leq M$.
It is well-known, see for example \cite[(22.12)]{woessBook}, that this defines an equivalence relation.
A~{\em hyperbolic boundary point} is an equivalence class of geodetic rays and the {\em hyperbolic boundary} $\rand G$ is the set of hyperbolic boundary points.
Let $\widehat{G}$ denote $G\cup\rand G$.

By~\cite[Proposition 7.2.9]{GhHaSur}, we can equip $\widehat{G}$ with a topology such that it is a compact space and such that every geodetic ray converges to the hyperbolic boundary point it is contained in.

Proposition~\ref{RightChoiceOfGeodRays} guarantees that we always find geodetic (double) rays with certain properties.

\begin{Prop}\label{RightChoiceOfGeodRays}\label{prop_geodExists} {\em \cite[(22.11) and (22.15)]{woessBook}}
Let $G$ be a hyperbolic graph with two distinct boundary points $\eta$ and $\nu$.
Let $o$ be a vertex in $G$, $x_1x_2\ldots$ a geodetic ray converging to $\eta$, and $y_1y_2\ldots$ a geodetic ray converging to $\nu$.
Then the following two properties hold:
\begin{enumerate}[{\em (i)}]
\item There is a geodetic ray in $G$ starting at $o$ and having only finitely many vertices outside $\{x_i\mid i\in\nat\}$.
\item There is a geodetic double ray $D$ having only finitely many vertices outside $\{x_i\mid i\in\nat\}\cup \{y_i\mid i\in\nat\}$.
One side of $D$ converges to $\eta$, the other to~$\nu$.\qed
\end{enumerate}
\end{Prop}

Equivalent geodetic rays stay close to each other:

\begin{prop}\label{prop_W22.12a}{\em \cite[Proposition 22.12]{woessBook}}
If $x_1x_2\ldots$ and $y_1y_2\ldots$ are equivalent geodetic rays in a hyperbolic graph, then there is a $k\in\ganz$ such that ${d(x_n,y_{n-k})\leq 2\delta}$ for all but finitely many~$n$.\qed
\end{prop}

Let $\gamma>1,c\geq 0$.
A  ray $x_0x_1\ldots$ in~$G$ is {\em $(\gamma,c)$-quasi-geodetic} if $d(x_i,x_j)\leq \gamma|i-j|+c$ for all $i,j\in\nat$. A $(\gamma,c)$-quasi-geodetic double ray is defined similarly.
Hence a (double) ray is geodetic, if it is a $(1,0)$-quasi-geodetic (double) ray.
If the constants $\gamma,c$ are not important then we just speak of {\em quasi-geodesics}.

The next proposition shows, that in every hyperbolic graph the geodesics and quasi-geodesics lie close to each other, see also \cite[Proposition 3.3]{ABCFLMSS}, \cite[3.1.3]{CoornDelPapa}, \cite[5.6, 5.11]{GhHaSur}, and \cite[7.2.A]{gromov}.

\begin{prop}\label{prop_geodAndQuasigeod}{\em\cite[Th\'eor\`eme~3.1.4]{CoornDelPapa}}
Let $G$ be a $\delta$-hyperbolic graph.
For all $\gamma_1\geq 1,\gamma_2\geq 0$ there is a constant $\kappa=\kappa(\delta,\gamma_1,\gamma_2)$ such that for every two vertices $x,y\in VG$ every $(\gamma_1,\gamma_2)$-quasi-geodesic between $x$ and $y$ lies in a $\kappa$-neighborhood around every geodesic between $x$ and $y$ and vice versa.

Furthermore, this extends to $(\gamma_1,\gamma_2)$-quasi-geodetic and geodetic rays as well as double rays.\qed
\end{prop}

The following result is~\cite[Proposition 3.2]{ABCFLMSS} (see also \cite[8.1.D]{gromov} and \cite[8.21]{GhHaSur}).

\begin{prop}\label{prop_ParaOnQuasiGeod}
Let $G$ be a transitive $\delta$-hyperbolic graph.
Let $x\in VG$ and $\alpha\in\Aut(G)$ be such that the orbit of~$x$ under $\alpha$ is infinite.
Then the set $\{\ldots,x{\alpha\inv},x,x\alpha,\ldots\}$ lies on a $(\kappa,\lambda)$-quasi-geodetic double ray for constants $\kappa\geq 1, \lambda\geq 0$ that depend only on $\delta$ and $d(x,x\alpha)$.\qed
\end{prop}

\subsection{Planar graphs}

A graph is {\em planar} if it admits an embedding, as a $1$-complex, into the Euclidean plane.
Such embeddings are called {\em planar} embeddings.

An embedding of $G$ is called {\em consistent} if, intuitively, it embeds every vertex in a similar way in the sense that the group action carries faces to faces. Let us make this more precise.
Given an embedding $\sigma$ of a graph $G$, we consider for every vertex $x$ the embedding of the edges incident with $x$, and define the {\em spin} of $x$ to be the cyclic order of the set $\{xy \mid y\in N(x)\}$ in which $xy_1$ is a successor of $xy_2$ whenever the edge $xy_2$ comes immediately after the edge $xy_1$ as we move clockwise around $x$. 

Call an automorphism $\alpha$ of $G$ {\em spin-preserving} if for every $x\in VG$ the spin of $x\alpha$ is the image of the spin of $x$ in $\sigma$. 
Call it {\em spin-reversing} if for every $x\in VG$ the spin of $x\alpha$ is the reverse of the image of the spin of $x$ in $\sigma$. Call an automorphism {\em consistent} if it is spin-preserving or spin-reversing in $\sigma$. Finally, call the embedding $\sigma$ {\em consistent} if every automorphism of $G$ is consistent in $\sigma$.

It is straightforward to check that $\sigma$ is consistent if and only if every automorphism of $G$ maps every facial path to a facial path. Thus the following classical result, proved by Whitney \cite[Theorem 11]{whitney_congruent_1932} for finite graphs and by Imrich \cite{ImWhi} for infinite ones, implies that all planar embeddings of a 3-connected transitive graph are consistent.

\begin{thm} \label{imrcb}
Let $G$ be a 3-connected graph embedded in the sphere. Then every automorphism of $G$ maps each facial path to a facial path. Thus every automorphism of $G$ is consistent.
\end{thm}

The next result is due to Babai and Watkins \cite{BW-Connectivity}, see also~\cite[Lemma~2.4]{B-GrowthRate}.

\begin{lem}\label{lem_BW}{\em \cite[Theorem~1]{BW-Connectivity}}
Let $G$ be a locally finite connected transitive graph  that has precisely one end.
Let $d$ be the degree of any of its vertices.
Then the connectivity of $G$ is at least $3(d+1)/4$.\qed
\end{lem}

We deduce from Lemma~\ref{lem_BW} and Theorem~\ref{imrcb} that every transitive planar graph with precisely one end has a consistent embedding in the Euclidean plane.
This means that for every transitive planar $1$-ended graph $G$ there are only two possibilities for the spin, one of which is the reverse of the other, such that every vertex of $G$ has one of these two spins.

\section{Proof of the main theorem}

We shall prove that every planar hyperbolic graph answers Question~\ref{quest_KaimWoess} in the negative.
Before we directly attack the question in the situation of planar graphs, we prove a general lemma (Lemma~\ref{lem_ellipticExists}) which might help to give a negative answer to the question in the general case.

\medskip

Let us recall the notions of elliptic and hyperbolic automorphisms.
Let $G$ be a hyperbolic graph.
\begin{enumerate}[(i)]
\item An automorphism of~$G$ is called {\em elliptic} if it fixes a finite set of vertices.
\item An automorphism $\alpha$ of~$G$ is called {\em hyperbolic} if it is not elliptic and fixes precisely two boundary points $\eta,\xi$ and if $(x{\alpha^n})_{n\in\nat}$ converges to~$\eta$ and $(x{\alpha^{-n}})_{n\in\nat}$ converges to~$\xi$ for every $x\in VG$.
\end{enumerate}

If $\alpha$ is a hyperbolic automorphism, then we call the boundary point to which all the sequences $(x{\alpha^n})$, $x\in VG$, converge the {\em direction} of~$\alpha$.

\medskip

For automorphism groups of hyperbolic graphs, there is the following classification of their elements; compare with \cite[Chapitre~9]{CoornDelPapa}.

\begin{lem}\label{lem_HypGraphNoParaElem}
Any automorphism of a hyperbolic graph is either elliptic or hyperbolic.\qed
\end{lem}

We now show the existence of certain elliptic and hyperbolic elements.

\begin{lem}\label{lem_ellipticExists}
Let $G$ be a $1$-ended $\delta$-hyperbolic graph and $\Gamma$ be a group acting transitively on~$G$ such that $\Gamma$ fixes a hyperbolic boundary point $\omega$ of~$G$.
Then the following statements hold.
\begin{enumerate}[{\em (i)}]
\item\label{item_hyperbolicElement} For every two vertices $x,y\in VG$ with $d(x,y)>2\delta$ that lie on a common geodetic double ray between $\omega$ and another hyperbolic boundary point, there exists a hyperbolic element $h$ in~$\Gamma$ with $xh=y$.
\item\label{item_oneElliptic} There exists a non-trivial elliptic element in~$\Gamma$ that fixes a vertex of~$G$.
\item\label{item_ellipticProduct} There exist two non-trivial distinct elliptic elements in~$\Gamma$ whose product is also non-trivial and elliptic and such that all these three automorphisms fix a common vertex of~$G$.
\end{enumerate}
\end{lem}

\begin{proof}
To prove (\ref{item_hyperbolicElement}) let $x,y$ lie on a common geodetic double ray $\pi$ as in the assertion with $d(x,y)=2\delta+d$ for a $d>0$ such that $x$ separates $y$ from $\omega$ on~$\pi$.
Let $\pi_y$ be the subray of~$\pi$ that starts at~$y$ and converges to~$\omega$.
As $\Gamma$ acts transitively on~$G$, there is an automorphism $\alpha\in\Gamma$ with $x\alpha=y$.
Lemma~\ref{lem_HypGraphNoParaElem} tells us that $\alpha$ is either hyperbolic or elliptic.
Let us suppose, for a contradiction, that $\alpha$ is elliptic.
Then the orbit of~$x$ under $\alpha$ is finite.
Let $n> 0$ be minimal with $x{\alpha^{n}}=x$.
We consider the rays $\pi_y{\alpha^i}$ with $i=0,\ldots,n$.
Each of these rays converges to~$\omega$ and contains the vertices $x{\alpha^{i+1}}$ and $x{\alpha^{i+2}}$.
As $G$ is $\delta$-hyperbolic, there is a vertex $z_1$ on $\pi_y\alpha$ with $d(x,z_1)\leq\delta$.
Since $d(x,x{\alpha})=2\delta+d$, the inequality $d(x,z_1)\leq\delta$ implies $d(x{\alpha},z_1)\geq \delta+d$.
So~$x$ has distance at most $\delta$ to a vertex on $\pi_y\alpha$ whose distance to $x\alpha^2$ is at least $3\delta+2d$.
Then there is a vertex $z_2$ on $\pi_y\alpha^2$ with distance at most $\delta$ to~$z_1$.
We have $d(x\alpha^2,z_2)\geq d(x\alpha^2,z_1)-d(z_1,z_2)\geq 2\delta+2d$ and, thus, $d(x\alpha^3,z_2)\geq 4\delta+3d$.
Inductively, $x=x\alpha^n$ lies in an $(n\delta)$-neighborhood of a vertex $z_{n}$ on $\pi_y\alpha^n$ whose distance on that ray to $x\alpha^{n+1}=x\alpha$ is at least $\delta+(n+1)(\delta+d)$.
But the inequality $d(z_{n},x\alpha)\geq \delta+(n+1)(\delta+d)$ implies $d(z_{n-1},x)\geq n(\delta+d)$ which is impossible.
Hence, $\alpha$ has to be hyperbolic, contradicting our assumption.

\medskip

For the proof of (\ref{item_oneElliptic}), let $\alpha_0$ be a hyperbolic element in~$\Gamma$.
Then $\alpha_0$ fixes $\omega$ and precisely one further boundary point $\eta_0$.
We assume that the direction of~$\alpha_0$ is~$\eta_0$.
For any $x_0\in VG$, there are constants $c_1\geq 1,c_2\geq 0$ such that the vertices $x_0{\alpha_0^i}$, $i\in\ganz$, lie on a $(c_1,c_2)$-quasi-geodetic double ray $\pi_0$ by Proposition~\ref{prop_ParaOnQuasiGeod}.
Note that $c_1$ and $c_2$ depend only on~$\delta$ and $d(x_0,x_0\alpha_0)$.

We are now going to construct a sequence $\sequ{x}$ of vertices in~$G$, a sequence $\sequ{\pi}$ of $(c_1,c_2)$-quasi-geodetic double rays, a sequence $\sequ{\eta}$ of hyperbolic boundary points, a sequence $\sequ{\alpha}$ of hyperbolic elements of~$\Gamma$, and a sequence $\sequ{\beta}$ of automorphisms of~$G$ such that the orbit of~$x_i$ under $\alpha_i$ lies on~$\pi_i$, such that the subrays of~$\pi_i$ converge either to~$\omega$ or to~$\eta_i$, such that $\eta_i$ is the direction of~$\alpha_i$, and such that $x_i$ has distance more than $2\kappa$ to all $\pi_j$ with $j<i$.
For this, let $\kappa=\kappa(\delta,c_1,c_2)$ be the constant from Proposition~\ref{prop_geodAndQuasigeod}, that is, every $(c_1,c_2)$-quasi-geodesic lies in a $\kappa$-neighborhood of a geodesic with the same endpoints and vice versa.
Let ${x_1\in VG}$ with $d(x_1,\pi_0)>2\kappa$ and let $\beta_1\in\Gamma$ with $x_0{\beta_1}=x_1$.
Then $\pi_1:=\pi_0{\beta_1}$ cannot lie $2\kappa$-close to a geodetic double ray between $\omega$ and~$\eta_0$.
Thus, we have $\eta_1:=\eta_0{\beta_1}\neq\eta_0$.
Since $\alpha_0$ is a hyperbolic element, so is $\alpha_1:=\beta_1\inv\alpha_0\beta_1$. Continuing like this we obtain the sequences as desired.
Among the automorphisms $\alpha_i$ and $\alpha_i\inv$ we shall find a pair the product of which is non-trivial, elliptic, and fixes some vertex as required by the assertion.

Consider an infinite sequence $\sequ{B}$ of balls of radius $2\kappa$ around elements of~$\pi_0$ that converge to~$\omega$.
Since $\Gamma$ acts transitively on~$VG$, there is a finite number $n$ such that each of these balls consists of~$n$ vertices.
As the number of the quasi-geodetic double rays with non-trivial intersection with $B_i$ increases with~$i$ and tends to infinity, there is a ball $B_m$ such that some vertex $b\in B_m$ lies on two distinct double rays $\pi_i,\pi_j$ with $i\neq j$.
Since $d(x_k,x_k\alpha_k)=d(x_0,x_0\alpha_0)$ for all $k\in\nat$ and since all balls of radius $d(x_0,x_0\alpha_0)$ have the same number of vertices, we may even assume that $b\alpha_i\inv=b\alpha_j\inv$.
Let us consider the automorphism $\alpha_i\inv\alpha_j$.
This automorphism obviously fixes~$b$, so it is an elliptic element, and it is non-trivial because of $\alpha_i\neq\alpha_j$.
This proves statement~(\ref{item_oneElliptic}).

\medskip

It remains to prove (\ref{item_ellipticProduct}).
We continue with the same notation as in the proof of~(\ref{item_oneElliptic}).
Let $\gamma:=\alpha_i\inv\alpha_j$ be the elliptic element we constructed in the proof of~(\ref{item_oneElliptic}).
Then, for each $k\in\nat$, $\gamma_k:=\alpha^k\gamma\alpha^{-k}$ is an elliptic element that is not trivial but acts trivially on~$b\alpha^{-k}$.
By a similar argument as above, we shall find two automorphisms of the $\gamma_k$ and $\gamma_k\inv$ that will satisfy together with their product the assertion (\ref{item_ellipticProduct}).
Each elliptic element $\gamma_k$ has to act on the set of $(c_1,c_2)$-quasi-geodetic rays from $b\alpha^{-k}$ to~$\omega$.
Let us consider the sequence of balls $\seq{D}{k}{\nat}$ with center $b\alpha^{-k}$ and radius $2\kappa$.
Like in the proof of (\ref{item_oneElliptic}), there is an $m\in\ganz$ such that two distinct $\gamma_k,\gamma_l$, with $k\neq l$, both fix a vertex~$y\in D_m$.
Then $\gamma_k\inv\gamma_l$ also fixes $y$ and it is again non-trivial because $\gamma_k\neq\gamma_l$.
Hence $\gamma_k\inv\gamma_l$ satisfies assertion~(\ref{item_ellipticProduct}).
\end{proof}

With this information about hyperbolic and elliptic elements in automorphism groups of hyperbolic graphs we can now prove our main result.

\begin{thm}\label{thm_FixedSetHyp}
For every planar $1$-ended hyperbolic graph $G$, and every group $\Gamma$ of automorphisms of~$G$ that acts transitively on~$VG$, no hyperbolic boundary point of~$G$ is fixed by all elements of $\Gamma$.
\end{thm}

\begin{proof}
Let us suppose, seeking for a contradiction, that there is a planar $1$-ended hyperbolic graph $G$ and a subgroup $\Gamma$ of~$\Aut(G)$ acting transitively on~$VG$ and fixing a hyperbolic boundary point $\omega$.
Let $\delta$ be the hyperbolicity constant of~$G$ as above, and let $d$ be the degree of some, and hence any vertex of~$G$.
Then we have $d\geq 3$ and, by Lemma~\ref{lem_BW} and Theorem~\ref{imrcb}, every automorphism in $\Gamma$ is consistent, either spin-preserving or spin-reversing.
Let $uvw$ be a 3-vertex subpath of a path $P$.
We say that a vertex $x\in N(v)\sm \{u,w\}$ lies {\em to the right} of~$P$ if in the spin of~$v$ we have $vx$ between $vw$ and~$vu$.
If $x$ does not lie to the right of~$P$ then it lies {\em to the left} of~$P$.

\medskip

By Lemma~\ref{lem_ellipticExists} there is a non-trivial elliptic element $\varphi\in\Gamma$.
If all non-trivial elliptic automorphisms are spin-reversing, then this is a direct contradiction to Lemma~\ref{lem_ellipticExists}\,(\ref{item_ellipticProduct}) as the product of any two spin-reversing automorphisms has to be spin-preserving.
Thus, we may assume that $\varphi$ is spin-preserving.

Let $y$ be a vertex with $y\varphi\neq y$.
As $\varphi$ is elliptic, there is a minimal $n\in\nat$ such that $y{\varphi^n}=y$.
For all $i=0,\ldots,n-1$, let $g_i$ be a geodesic from $y_{i}:=y{\varphi^{i}}$ to~$y_{i+1}$ such that $g_i\varphi=g_{i+1}$ and let $\pi_i$ be a geodetic ray from $y{\varphi^i}$ converging to~$\omega$ such that $\pi_i\varphi=\pi_{i+1}$.
Then $C:=g_0\ldots g_{n-1}$ is a cycle of length $n\cdot l(g_0)$.
We distinguish two cases that will both lead to a contradiction: either $\pi_i$ and $\pi_{i+1}$ intersect infinitely often or not.

\medskip

Let us first consider the case that they have only finitely many common vertices.
By choosing some other vertex $y'$ on~$\pi_0$ instead of~$y$ we may assume that the corresponding rays $y_i\pi_i$ and $y_{i+1}\pi_{i+1}$ have no common vertex.
In this situation, we shall show that any hyperbolic boundary point but~$\omega$ is separated from $\omega$ by some finite cycle, which is impossible as $G$ has precisely one end.

For every vertex on~$\pi_0$ with distance larger than $l(g_0)+\delta$ to~$y_0$ there is a vertex on~$\pi_1$ of distance at most~$\delta$.
We fix a geodesic between each such two vertices whose intersection with $\pi_0$, $\pi_1$, respectively, is a connected subpath.
If we consider an infinite sequence $\sequ{z}$ of vertices of~$\pi_0$ with strictly increasing distance to~$y_0$, then either there are two infinite subsequences such that the chosen geodesics for one of them always lie to the right of~$\pi_0$ for each vertex and for the other sequence always to the left of~$\pi_0$, or the geodesics for all but finitely many lie to the same side, say to the right of~$\pi_0$.
If the first of these two situations occurs, then any geodetic ray either is equivalent to~$\pi_0$---and hence converges to~$\omega$---or is separated by some finite cycle from~$\omega$, which is impossible as $G$ is $1$-ended.
Thus, we may assume that all the above described geodesics lie eventually to the right of~$\pi_0$.
Let $V_0$ be a subset of~$VG$ consisting of all the vertices from the paths $\pi_0$, $\pi_1$, $g_0$ and from all the paths from $\pi_0$ whose first vertex lies to the right of~$\pi_0$ and that has only vertices not in $\pi_0\cup\pi_1\cup g_0$ except for its first vertex.
This means that $V_0$ consists of all those vertices of~$G$ that are separated in the plane by $\pi_0\cup\pi_1\cup g_0$ from any vertex that lies to the right of~$\pi_1$.
Then any ray in~$G[V_0]$ has to converge to~$\omega$.
Similarly we find $V_1,\ldots,V_{n-1}$, always taking the vertices to the right of~$\pi_i$ to obtain $V_i$, because $\varphi$ is spin-preserving.
We conclude that any other hyperbolic boundary point is separated from~$\omega$ by~$C$, which is impossible since $G$ has precisely one end.

\medskip

Thus, the only case left is that $\pi_0$ and~$\pi_1$ have infinitely many common vertices.
By Proposition~\ref{prop_W22.12a}, there is a $k\in\nat$ such that for all but finitely many vertices $x$ on~$\pi_0$ we have $d(x,x\varphi)\leq k+2\delta$.
Again we distinguish two cases: either $\pi_0-\pi_1$ contains infinitely many vertices or only finitely many.
Let us first suppose that there are infinitely many vertices in $\pi_0-\pi_1$.
Let $\seq{x}{i}{\nat}$ be a sequence of pairwise distinct vertices on $\pi_0\cap\pi_1$ such that the predecessor of~$x_i$ on~$\pi_0$ does not lie on~$\pi_1$ and such that the predecessor of~$x_i$ on~$\pi_1$ lies always to the same side of~$\pi_0$ at the vertex $x_i$, say to the right.
Then there is an~$M\in\nat$ such that we have for all $i\geq M$ that $C_i:=x_i\pi_1 x_i\varphi\pi_2\ldots x_i{\varphi^{n-1}}\pi_0 x_i$ is a cycle separating $C$ from~$\omega$.
The inequality $d(x,x\varphi)\leq k+2\delta$ immediately implies that all these cycles have length at most $2n(k+2\delta)$ because of $d(x_i,x_i{\varphi^{n-1}})\leq n(k+2\delta)$.
Now consider the ball $B$ with center $x_1$ and radius $2n(k+2\delta)$.
As $G$ is locally finite, $G-B$ has only finitely many components and precisely one of them is infinite because $G$ is $1$-ended.
Let $N$ denote the number of vertices in finite components of $G-B$.
Then we look at any ball $B'$ with center $x_i$ for an $i>N+2n(k+2\delta)+1$ and radius $2n(k+2\delta)$.
Again, $G-B'$ has only one infinite component and the number of vertices in the finite components of $G-B'$ is precisely $N$ by the transitivity of~$G$.
But because of~$C_i\sub B'$, the component $A$ that contains $x_1$ is finite.
Since $\pi_0$ is geodetic, $A$ contains all $x_j$ with $j\leq N+1$, so there are at least $N+1$ vertices in finite components of $G-B'$ which is a contradiction to the transitivity of~$\Gamma$ on~$G$.

Thus, the only remaining case is when there are only finitely many vertices in $\pi_0-\pi_1$.
By replacing $y$ by another suitable vertex on~$\pi_0$, we may assume that all the vertices of~$\pi_0$ lie on~$\pi_1$ or vice versa.
But then either$$\pi_n=y_0\pi_{n}y_{n-1}\ldots y_1\pi_1 y_0\pi_0$$
or
$$\pi_0=y_0\pi_0y_1\pi_1\ldots y_{n-1}\pi_{n-1}y_0\pi_n$$
contains a cycle and so it cannot be a geodetic ray, contrary to our assumption.
This completes the proof.
\end{proof}

\providecommand{\bysame}{\leavevmode\hbox to3em{\hrulefill}\thinspace}
\providecommand{\MR}{\relax\ifhmode\unskip\space\fi MR }
\providecommand{\MRhref}[2]{%
  \href{http://www.ams.org/mathscinet-getitem?mr=#1}{#2}
}
\providecommand{\href}[2]{#2}

\end{document}